\documentclass[12pt]{article}
\usepackage{epsfig,graphics,graphicx}
\usepackage{amssymb}
\usepackage{amsmath}
\usepackage{extarrows}
\usepackage[utf8]{inputenc}
\usepackage[english]{babel}
\usepackage{geometry}
\usepackage{amsthm}
\usepackage{scalerel}[2014/03/10]
\usepackage[usestackEOL]{stackengine}
\usepackage{esint}
\usepackage{xcolor}

\newtheorem{theo}{Theorem} 

\newtheorem{lem}[theo]{Lemma}
\newtheorem{prop}[theo]{Proposition}

\newtheorem{rem}[theo]{Remark}

\def\R{\mathbb R}

\def\D{\mathbb D}
\def\C{\mathbb C}

\def\bC{\mathbf C}

\def\bM{\mathbf M}

\def\Re{\text{Re}}
\def\Im{\text{Im}}

\def\diam{\text{diam}}
\def\bv{\mathbf{v}}

\def\div{\operatorname{div}}
\def\supp{\operatorname{supp}}

\def\curl{\operatorname{curl}}

\def\v{\operatorname{\textbf{v}}}

\title{\sc{Nonlinear transport equations \\ and quasiconformal maps}}
\author{\sc{Albert Clop, Banhirup Sengupta}}
\date{}

\begin{document}

\maketitle

\abstract{We prove existence of solutions to a nonlinear transport equation in the plane, for which the velocity field is obtained as the convolution of the classical Cauchy Kernel with the unknown. Even though the initial datum is bounded and compactly supported, the velocity field may have unbounded divergence. The proof is based on the compactness property of quasiconformal mappings.}

\section{Introduction}

\noindent
In this article we prove existence of global in time solutions to the following active scalar equation,
\begin{equation}\label{cauchysystem}
\begin{cases}
\frac{d}{dt}\omega+\v\cdot \nabla \omega = 0, \\
\v(t,\cdot)=K\ast \omega(t,\cdot),\\
\omega(0,\cdot)=\omega_0.
\end{cases}\end{equation}
In the above system, one has 
$$\aligned
K(z)=\frac{e^{i\theta}}{2\pi z}=\frac1{2\pi}\,\frac{(x\,\cos\theta+y\,\sin\theta, x\,\sin\theta-y\,\cos\theta) }{x^2+y^2},
\endaligned$$
and $\theta\in [0, 2\pi]$ is fixed, while $\omega_0\in L^\infty$ is a given compactly supported and real valued function. This model arises as a natural counterpart to the classical \emph{planar Euler system of equations in vorticity form}, which is given also by \eqref{cauchysystem} but with a different choice for the kernel $K$, namely
$$\aligned
K(z)=\frac{i}{2\pi\,\bar{z}}=\frac1{2\pi}\,\frac{(-y,x)}{ x^2+y^2}.
\endaligned$$
In both cases, the quantity $\partial_{t}+\v\cdot \nabla$ is called the \emph{material derivative} of the unknown $\omega:[0,\infty)\times \C\to\R$, and $\v$ is called the \emph{velocity}. In Euler system, $\v$ represents the velocity field of a perfect, incompressible, inviscid fluid, and $\omega$ is known as the \emph{vorticity of the fluid}. \\
\\
In Euler's setting, the Biot-Savart law $\v = K\ast \omega$ can be written in terms of complex derivatives as
\begin{equation}\label{dv}
\partial \v = \frac{i\omega}2.
\end{equation}
where $\partial=\frac{\partial_x-i\partial_y}{2}$ denotes the classical complex derivative. Since $\omega$ is real valued, this ensures that $\v$ has divergence $\div(\v)=0$, and its curl is $\curl(\v)=\omega$. On the other hand, and under enough regularity, the transport structure of the equation guarantees for the solution the following representation formula,
\begin{equation}\label{representation}
\omega(t,X(t,x))=\omega_0(x),
\end{equation}
where $X(t,x)$ is the flow of $\bv$, that is, the solution to the ODE
$$\begin{cases}\frac{d}{dt}X(t,x)=\bv(t,X(t,x))\\X(0,x)=x\end{cases}$$
The incompressibility condition, i.e. $\div(\v)=0$, guarantees $X(t,\cdot)$ to be a measure preserving self map of $\R^2$, and so the $L^1$ norm of $\omega(t,\cdot)$ is constant in time. On the other hand, if $\omega_0$ is essentially bounded, then $\omega$ is essentially bounded as well, and so the $L^\infty$ norm of $\omega$ is also constant in time. So both the incompressibility of the fluid and the boundedness of $\curl(\v)$ are essential to get for $\omega$ uniform $L^1$ and $L^\infty$ bounds. These bounds are basic in the proof of Yudovich's Theorem \cite{Y}, which establishes existence and uniqueness of global in time solutions to the Euler system under the assumption $\omega_0\in L^\infty$. \\
\\
In contrast to \eqref{dv}, in our new setting \eqref{cauchysystem} the kernel ensures now that 
\begin{equation}\label{dbarv}
\overline\partial \v =\frac{e^{i\theta}\,\omega}{2}
\end{equation}
where $\overline\partial=\frac{\partial_x+i\partial_y}{2}$ denotes the anticonformal complex derivative. Especially, $\div(\v)$ needs not be identically $0$, so the vector field $\v$ is not anymore incompressible. Moreover, classical Calder\'on-Zygmund Theory can be used to show that now, even for bounded and compactly supported $\omega_0$, both $\div(\v)$ and $\curl(\v)$ may be unbounded functions. Still the transport structure of the equation is unaffected by the change on the kernel, and nice solutions $\omega(t,\cdot)$ admit again the representation formula \eqref{representation}, although now the flow $X(t,\cdot)$ needs not be measure preserving. As a consequence, the control in time of both $L^1$ and $L^\infty$ norms of $\omega(t,\cdot)$ is not so automatic, and might even fail.  \\
\\
For certain linear transport models \cite{CJMO, CJMO2, CJMO3}, it has been recently shown  that their well-posedness do not depend on the measure-preservation property of the flow and, instead, the preservation of Lebesgue null sets is the only requirement. Such models already show that Lebesgue null sets may be preserved by the flow if the velocity field has non-zero or even unbounded divergence.\\
\\
In the same way $\|\partial\v\|_{L^\infty}$ keeps bounded in time for any Yudovich solution to the Euler system, in our setting \eqref{cauchysystem} the quantity $\|\overline\partial\v\|_{L^\infty}$ keeps bounded in time as long as one is able to show the preservation of Lebesgue null sets, rather than the preservation of Lebesgue measure through the flow. Having uniform bounds for $\|\overline\partial\v\|_{L^\infty}$ immediately drives our attention to H.M. Reimann's paper \cite{R}. There it was shown that such vector fields produce flows $X(t,x)$ with the very special property of being quasiconformal for every $t>0$. Quasiconformal maps are known to Geometric Function Theory experts to be a very well understood class of homeomorphisms, and their compactness properties make them specially suitable for solving certain elliptic PDE problems. This time, though, we will use them for a purely hyperbolic PDE. Our main result is as follows.
\begin{theo}\label{main}
If the initial datum $\omega_{0}\in L^\infty$ has compact support, there exists a solution $\omega\in L^{1}([0,T], L^\infty)$ of \eqref{cauchysystem}
for every $T>0$.
 \end{theo}
 \noindent

The paper is structured as follows.  In Section \ref{measurabletheory} we prove existence of solutions of \eqref{cauchysystem}.  In Section \ref{governing} we show that \eqref{cauchysystem} admits an equivalent formulation in terms of $\v$ and a scalar valued function $q$, similar to the velocity formulation of Euler's system, and explain why this formulation fails to provide uniqueness of solutions.\\
\\
{\bf{Acknowledgements}}. The authors were partially supported by projects $2017SGR395$ (Govt. of Catalonia) and $ PID2020-112881GB-I00$ (Govt. of Spain).

\section{Existence  theory for $\omega_0\in L^\infty$}\label{measurabletheory}

\noindent
Given compactly supported $\omega_0\in L^\infty(\C)$, we look for scalar-valued functions $\omega:\R\times\C\to\R$ belonging to $L^1(\R, L^\infty(\C))$ that solve the problem \eqref{cauchysystem}. Our goal is to prove that a weak solution to \eqref{cauchysystem} exists and can be represented by
$$\omega(t,X(t,z))=\omega_0(z)$$
where $X$ are the trajectories of the vector field $\v$. To this end, we start by mollifying the datum $\omega_0$ to $\omega_0^\epsilon\in C^\infty$ in such a way that
$$\aligned
\|\omega_0^\epsilon\|_\infty&\leq \|\omega_0\|_\infty,\\
\|\omega_0^\epsilon\|_1&\leq \|\omega_0\|_1, \text{ and}\\
\omega_0^\epsilon\text{  has comp}&\text{act support},
\endaligned$$
and moreover $\|\omega_0^\epsilon-\omega_0\|_1\to 0$, as $\epsilon\to 0$. Then, by virtue of the smooth theory (see for instance \cite[Theorem 2]{C}), to each $\omega_0^\epsilon$ we can associate its unique solution $\omega^\epsilon$ to 
\begin{equation}\label{nlteps}\begin{cases}
\frac{d}{dt}\,\omega^\epsilon +\v^\epsilon\cdot\nabla\omega^\epsilon=0,\\
\v^\epsilon(t,\cdot)=K\ast\omega^\epsilon(t,\cdot),\\\omega^\epsilon(0,\cdot)=\omega^\epsilon_0\end{cases}
\end{equation}
with $K(z)=\frac{e^{i\theta}}{2\pi z}$. For each $t\in\R$, $\omega^\epsilon(t,\cdot)$ is continuous and compactly supported, locally uniformly in time. As a consequence, the velocity field $\v^\epsilon=K\ast\omega^\epsilon$ is at least $C^1$ in the space variable. This guarantees existence and uniqueness of a well defined flow $X^\epsilon$ of diffeomorphisms $X^\epsilon(t,\cdot):\C\to\C$ such that
$$\begin{cases}\frac{d}{dt}X^\epsilon (t,z)=\v^\epsilon(t,X^\epsilon(t,z)),\\X^\epsilon(0,z)=z.\end{cases}$$
Moreover, the solution $\omega^\epsilon$ is obtained by translating the datum $\omega_0^\epsilon$ along the trajectories $X^\epsilon(t,x)$ of $\v^\epsilon$, that is, $$\omega^\epsilon(t,X^\epsilon(t,z))=\omega_0^\epsilon(z).$$
Having $\div(\v^\epsilon)\in L^\infty$, it is clear that $X^\epsilon(t,\cdot)$ preserves Lebesgue null sets. Thus, and since $\omega_0^\epsilon\in L^\infty(\C)$, we have $\|\omega^\epsilon(t,\cdot)\|_\infty=\|\omega_0^\epsilon\|_\infty$.\\
\\
The proof of Theorem \ref{main} consists of proving convergence of the solutions $\omega^\epsilon$ and $\v^\epsilon$ in an appropriate sense. As usually, the most delicate point is the following uniform $L^1$ bound,
$$\|\omega^\epsilon(t,\cdot)\|_1\leq e^{2t\,\|\omega_0\|_\infty}\,\|\omega_0\|_\infty\,|\supp\omega_0|.$$
This, combined with the preservation in time of $\|\omega^\epsilon(t,\cdot)\|_\infty$ and Lemma \ref{LinftyK} below, gives uniform bounds for $\v^\epsilon$, which are essential to find limit trajectories. \\
\\
In the classical Euler's setting, that is for the kernel $K(z)=\frac{i}{2\pi\bar{z}}$, the (uniform in $\epsilon$) $L^1$ control of $\omega^\epsilon(t,\cdot)$ is automatic from the measure-preserving property of the flow, namely $\det DX^\epsilon(t,\cdot))=1$. In turn, this comes from the fact that $\div(\v^\epsilon)(t,\cdot)=0$ at every time $t$. Now, for $K(z)=\frac{e^{i\theta}}{2\pi z}$, one certainly has $\div(\v^\epsilon)(t,\cdot)\in L^\infty$ at any time, due to the smoothness of $\v^\epsilon$, but as $\epsilon\to 0$ one might see $\|\div(\v^\epsilon)(t,\cdot)\|_\infty$ blowing up, and with it any uniform bound on $\det DX^\epsilon(t,\cdot)$ would also blow up. It is very remarkable that under these circumstances still the uniform $L^1$ control of $\omega^\epsilon$ is possible, and comes as a consequence of the fact that the flow consists of principal quasiconformal maps which are conformal outside of the support of $\omega_0$, and moreover with uniformly bounded distortion.  To show this, step by step, we first need to recall the following result, due to H.M. Reimann \cite{R}. We only state it on the plane, although it holds also in higher dimensions.

\begin{theo}\label{Reimannth}
Let $\bv:[0,T]\times\C\to\C$ be a continuous vector field, such that for each $t$ one has
$$\limsup_{|z|\to+\infty}\frac{|\bv(t,z)|}{|z|\,\log|z|}<+\infty.$$
Suppose that the distributional derivatives $\partial \bv(t,\cdot)$ and $\overline\partial \bv(t,\cdot)$ are locally integrable functions of $z\in\C$, and moreover suppose that 
$$\sup_{t\in[0,T]}\|\overline\partial\bv (t,\cdot)\|_\infty \leq C_0<\infty.$$
Then, $\bv$ admits a unique flow $X(t,z)$ of $K_t$-quasiconformal maps $X(t,\cdot):\C\to\C$, and
$$K_t\leq \exp \left(2\int_0^t \|\overline\partial \bv(s,\cdot)\|_\infty \,ds\right).$$
\end{theo}

\noindent
We wish to remark here the existence of a counterpart to this theorem, with $\overline\partial \v$ replaced by $\overline\partial\v+\lambda\,\Im(\partial \v)$, where one may choose $\lambda$ to be a constant $\lambda\in\D$ or also a smooth, compactly supported function with $\|\lambda\|_\infty<1$. The change in the operator may result in a change in the bounds for $K_t$ as well. See \cite[Theorem 1]{CJ} for more details. This counterpart may produce extensions to Theorem \ref{main}, as we will explain later on.\\
\\
We will be using also the following elementary properties of the convolution with $K(z)=\frac{e^{i\theta}}{2\pi z}$. 

\begin{lem}\label{LinftyK}
Let $f:\C\to\C$ be given, and assume $f\in L^\infty$. 
\begin{itemize}
\item[(a)] If $f\in L^{1}$, then $K\ast f\in L^\infty$ and 
\begin{equation}\label{without support}
\|K\ast f\|_\infty\leq C\,\|f\|_1^{\frac12}\,\|f\|_\infty^{\frac12},
\end{equation} 
\item[(b)] If $f$ has compact support, then
\begin{equation}\label{with support}
\|K\ast f\|_\infty\leq C\,|\supp f|^{\frac12}\,\|f\|_\infty.
\end{equation}
\item[(c)] If $f$ is compactly supported, then 
$$\limsup_{|x|\to\infty} |x|\,|K\ast f(x)|\leq C<+\infty$$
with $C$ depending only on $\|f\|_\infty$ and $|\supp(f)|$. 
\end{itemize}
\end{lem}
\begin{proof}
Let us consider a real number $R>0$. For each such $R$ and given any two arbitrary points $x$ and $y$ in the plane, we can always divide the plane into two regions, $|x-y|\leq R$ and $|x-y|>R$. Therefore,
$$\aligned
|K\ast f(x)|
&\leq\int_{\C}|K(x-y)\,f(y)|\,dA(y)\\
&=\int_{|x-y|\leq R}|K(x-y)\,f(y)|\,dA(y)+\int_{|x-y|>R}|K(x-y)\,f(y)|\,dA(y)\\
&\leq\|f\|_\infty\,\int_{|x-y|\leq R}\frac{C}{|x-y|}\,dA(y)+\frac{C}{R}\int_{|x-y|>R}|f(y)|dy\\
&\leq CR\,\|f\|_{\infty}+\frac{C}{R}\|f\|_1
\endaligned$$ 
If we minimize the term on the right hand of the inequality as a function of $R$, the best possible value attainable is $R=\|f\|_1^\frac12\,\|f\|_\infty^{-\frac12}.$ This gives the bound \eqref{without support}. The bound \eqref{with support} is an immediate consequence of \eqref{without support}. Concerning (c), let us assume that $diam(\supp(f))=2R$. It is not restrictive to assume $0\in\supp(f)$, so that $\supp(f)\subset \overline{D(0,R)}$. Then, at points $x$ such that $|x|>2R$ one has
$$\aligned
|K\ast f(x)|
&\leq C\int_{|y-x|\leq R}\frac{|f(y-x)|}{|y|}\,dA(y)\\ 
&\leq C\|f\|_\infty\,\int_{|y-x|\leq R}\frac{1}{|y|}\,dA(y)\\ 
&\leq C\|f\|_\infty\,\frac{R^2}{|x|-R}\leq C\,\frac{\|f\|_\infty\,\diam(\supp(f))^2}{|x|}, \\ 
\endaligned$$
as claimed.
\end{proof}

\noindent
Next, in order to proceed with the proof of Theorem \ref{main}, we start with a Lemma. 

\begin{lem}\label{qcflows}
Let $T>0$ be fixed. Then: 
\begin{itemize}
\item[(a)] $X^\epsilon(t,\cdot)$ is $\mathcal{K}_t$-quasiconformal, with $1\leq\mathcal{K}_t\leq e^{|t|\|\omega_0\|_\infty}$, where $-T\leq t\leq T$.
\item[(b)] There exists a constant $C=C(\mathcal{K}_t)$ such that
$$
\frac1C\left(\frac{|z-z_0|}{|z-w_0|}\right)^{\mathcal{K}_t}\leq\frac{|X^\epsilon(t,z)-X^\epsilon(t,z_0)|}{|X^\epsilon(t,z)-X^\epsilon(t,w_0)|}\leq C\left(\frac{|z-z_0|}{|z-w_0|}\right)^\frac1{\mathcal{K}_t}
$$
for any $z,z_0,w_0\in\C$ and any time $t\in[-T,T]$. 
\item[(c)] There exists a constant $C=C(\mathcal{K}_t)$ such that  
$$\frac{|X^\epsilon(t,E)|}{|X^\epsilon(t,D)|}\leq C(\mathcal{K}_t)\,\left(\frac{|E|}{|D|}\right)^\frac1{\mathcal{K}_t}.$$
whenever $D\subset\C$ is a disk and $E\subset D$ is measurable.
\end{itemize}
\end{lem}
\begin{proof}
The structure of the Cauchy Kernel makes it clear that
$$2\|\overline\partial{\v}^\epsilon(t,\cdot)\|_\infty=  \|\omega^\epsilon(t,\cdot)\|_\infty=\|\omega^\epsilon_0\|_\infty\leq\|\omega_0\|_\infty.$$
Moreover, from Lemma \ref{LinftyK} (c) we know that $\v^{\epsilon}(t,\cdot)$ vanishes at $\infty$ like $\frac{C}{|z|}$. Therefore,
$$\aligned
\underset{|z|\to\infty}{\limsup}\frac{|\v^{\epsilon}(t,z)|}{|z|\log\left(e+|z|\right)}&\leq C<+\infty.
\endaligned$$
Thus, all the requirements in Reimann's Theorem \ref{Reimannth} are fulfilled, and quasiconformality follows, with quasiconformality constant 
$$\mathcal{K}_t\leq \exp\left(2\int_0^t\|\overline\partial \v^\epsilon(s,\cdot)\|_\infty ds\right)\leq e^{t\|\omega_0\|_\infty}$$
and by definition, $\mathcal{K}_t\geq1$. Therefore, part (a) is clear. Part (b)  says that quasiconformal maps are quantitatively quasisymmetric. The interested reader should check \cite[Corollary 3.10.4]{AIM} for a detailed proof. Part (c) follows from \cite[Theorem 13.1.5]{AIM} and the classical area distortion estimates for $\mathcal{K}_t$-quasiconformal maps.
\end{proof}

\noindent
Next, we would like to find an \emph{accumulation point} $X(t,\cdot)$ of the \emph{trajectories} $X^\epsilon(t,\cdot)$. As always, this will be done by using the \emph{control in time} of the $L^1$ norm of $\omega^\epsilon$. However, in contrast to Euler's situation, this time the control will be obtained in a completely different way. As a first step, let us note that \emph{compactness of the flow} will be a direct consequence of \emph{local boundedness}. 

\begin{lem}\label{compactness}
Assume that $X^\epsilon(t,\cdot)$ is uniformly bounded on compact sets. Then:
\begin{enumerate}
\item[(a)] $\{X^\epsilon(t,\cdot)\}_\epsilon$ is pointwise equicontinuous.
\item[(b)] $\{X^\epsilon(t,\cdot)\}_\epsilon$ accumulates to a $\mathcal{K}_t$-quasiconformal map $X(t,\cdot)$.
\end{enumerate}
\end{lem}
\begin{proof}
To prove the claim (a), let us remind from Lemma \ref{qcflows} that $X^\epsilon(t,\cdot)$ is quasisymmetric. That is, given any three points $z_0,z,w\in\C$ we have
$$
\frac{|X^\epsilon(t,z )-X^\epsilon(t,z_0)|}{|X^\epsilon(t,w)-X^\epsilon(t,z_0)|}\leq \eta_{\mathcal{K}_t}\left(\frac{|z -z_0|}{|w-z_0|}\right).
$$
As a consequence
$$\aligned
 |X^\epsilon(t,z )-X^\epsilon(t,z_0)| 
 &\leq \eta_{\mathcal{K}_t}(|z-z_0|/|w-z_0|) |X^\epsilon(t,w)-X^\epsilon(t,z_0)|\\
 &\leq \eta_{\mathcal{K}_t}(|z-z_0|/|w-z_0|) (|X^\epsilon(t,w)|+|X^\epsilon(t,z_0)|)\\
 &\leq \eta_{\mathcal{K}_t}(|z-z_0|/|w-z_0|) ( C(t,|w|) +C(t,|z_0|))
\endaligned$$
In particular, by leaving $w$ fixed one can easily get that $X^\epsilon(t,\cdot)$ is equicontinuous at $z_0$. \\
The family of maps, $\{X^\epsilon(t,\cdot)\}_\epsilon$ is pointwise equicontinuous and locally uniformly bounded. Therefore, Arzela-Ascoli theorem ensures the existence of a locally uniform accumulation point $X(t,\cdot)$. It is worth mentioning that by classical tools in \emph{Geometric Function Theory} \cite[Theorem 3.1.3]{AIS} it can only be either $\mathcal{K}_t$-quasiconformal or constant. To see that it cannot be constant, one must observe that the quasisymmetry bounds are preserved by uniform limits. Being two sided, these quasisymmetry bounds guarantee bijectivity. Therefore, the accumulation point $X(t,\cdot)$ is $\mathcal{K}_t$-quasiconformal. 
\end{proof}


\noindent
In order to get the local boundedness of the flow, the key point is the following elementary fact.

\begin{lem}\label{support}
Let $X^\epsilon(t,\cdot)$ be as before, and assume that $\omega_0^\epsilon$ has compact support. Then 
$$\omega^\epsilon_0(z)=0\hspace{1cm}\Longrightarrow\hspace{1cm} \overline\partial X^\epsilon(t,z)=0,$$
in other words $X^\epsilon(t,\cdot)$ is conformal outside of $\supp\omega_0^\epsilon$.
\end{lem}
\begin{proof}
The $\mathcal{K}_t$-quasiconformality of $X^\epsilon(t,\cdot)$ ensures the existence of a well-defined, uniformly elliptic Beltrami coefficient $\mu^\epsilon(t,\cdot)=\frac{\overline\partial X^\epsilon(t,\cdot)}{\partial X^\epsilon(t,\cdot)}$, and moreover we know that 
$$\|\mu^\epsilon(t,\cdot)\|_\infty\leq\frac{\mathcal{K}_t-1}{\mathcal{K}_t+1}.$$
The smoothness in time of $\partial X^\epsilon(t,z)$ and $\overline\partial X^\epsilon(t,z)$ guarantees that $t\mapsto \mu^\epsilon(t,z)$ is also smooth. From the equation for the flow $\dot{X^\epsilon}(t,z)=\v^\epsilon(t,X^\epsilon(t,z))$ and the chain rule we get that
$$\aligned
\overline\partial \v^\epsilon(t,X^\epsilon(t,z))
&=\frac{\frac{d}{dt}\overline\partial X^\epsilon(t,z)\,\partial X^\epsilon(t,z)-\overline\partial X^\epsilon(t,z)\,\frac{d}{dt}\partial X^\epsilon(t,z)}{J^\epsilon(t,z)}\\
&=\frac{\frac{d}{dt}\mu^\epsilon(t,z)\,\left(\partial X^\epsilon(t,z)\right)^2}{J^\epsilon(t,z)}\\
&=\frac{\frac{d}{dt}\mu^\epsilon(t,z)}{1-|\mu^\epsilon(t,z)|^2}\,\frac{\partial X^\epsilon(t,z)}{\overline{\partial X^\epsilon(t,z)}}\\
\endaligned$$
On the other hand, from the kernel structure we have
$$
2|\overline\partial \v^\epsilon(t,X^\epsilon(t,z))|= |\omega^\epsilon(t,X^\epsilon(t,z))|=|\omega^\epsilon_0(z)|.
$$
Thus
$$
\aligned
\frac{ \frac{d}{dt}|\mu^\epsilon(t,z)|}{1-|\mu^\epsilon(t,z)|^2}
&\leq\frac{\left|\frac{d}{dt}\mu^\epsilon(t,z)\right|}{1-|\mu^\epsilon(t,z)|^2}=\frac12\,|\omega^\epsilon_0(z)|
\endaligned$$
Now, given any time $t>0$, we can integrate on $(0,t)$ the above inequality to obtain that
\begin{equation}\label{pointw}
\log\left(\frac{1+|\mu^\epsilon(t,z)|}{1-|\mu^\epsilon(t,z)|}\right)\leq t\,|\omega^\epsilon_0(z)|,
\end{equation}
since $X^\epsilon(0,z)=z$ implies $\mu^\epsilon(0,z)=0$. Now, if $\omega^\epsilon_0(z)=0$ then necessarily $\mu^\epsilon(t,z)=0$ and hence $\overline\partial X^\epsilon(t,z)=0$. The claim follows.
\end{proof}

\begin{rem}
The above proof also shows that, at time $t=0$,
$$\frac{\omega^{\epsilon}_{0}(z)}{2}=\frac{1}{2}\,\omega^\epsilon(0,\cdot)(z)=\overline\partial \bv^\epsilon(0, z)=\frac{d}{dt}\left[\mu^\epsilon(t,z)\right]_{t=0}.$$
That is, the initial vorticity determines the time derivative of the Beltrami coefficient at time $t=0$. Thus, it is natural to ask for the dependence of $X^\epsilon(t,\cdot)$ under second-order perturbations of $\mu^\epsilon(t,z)$.
\end{rem}

\noindent
Now, it just remains to observe that $\v^\epsilon(t,\cdot)$ cannot grow without control as $|z|\to\infty$. This, together with the conformality of the flow outside of $\supp\omega^\epsilon_0$, provides improved area estimates which are essential for the control of $\|\omega^\epsilon(t,\cdot)\|_1$.

\begin{lem}\label{principal}
Let $X^\epsilon(t,\cdot)$ be as before, and assume that $\omega_0$ has compact support.
\begin{itemize}
\item[(a)] For each $t,\epsilon$ there exists $b^\epsilon(t)\in\C$ such that $\lim_{|z|\to\infty}|X^\epsilon(t,z)-z-b^\epsilon(t)|=0.$
\item[(b)] One has $|X^\epsilon(t,E)|\leq \mathcal{K}_t\,|E|$ for each set $E\supset \supp\omega_0^\epsilon$.
\end{itemize}
\end{lem}
\begin{proof}
By Lemma \ref{support} we know that $X^\epsilon(t,\cdot)$ is conformal on a neighborhood of $\infty$. Therefore, it has around $\infty$ a Laurent series development whose higher order term is linear,
$$
X^\epsilon(t,z)=a^\epsilon(t)z + b^\epsilon(t)+\frac{c^\epsilon(t)}{z}+...
$$
Also, from Lemma \ref{LinftyK} (b) and the integral representation of $X^\epsilon(t,\cdot)$, we know that
$$\aligned
 |X^\epsilon(t,z)-z| 
&= \left|\int_0^t\v^\epsilon(s,X^\epsilon(s,z))\right| \,ds\\
&\leq \int_0^t \left|\v^\epsilon(s,X^\epsilon(s,z))\right|\,ds\\
&\leq \int_0^tC(K)\,\|\omega^\epsilon(s,\cdot)\|_\infty\,|\supp \omega^\epsilon(s,\cdot)|^\frac12\,ds\\
&\leq \int_0^t C(K)\,\|\omega^\epsilon_0\|_\infty\,|X^\epsilon(s,\supp \omega^\epsilon_0)|^\frac12 \,ds\\
&\leq C(K)\,\|\omega^\epsilon_0\|_\infty \,\int_0^t |X^\epsilon(s,\supp \omega^\epsilon_0)|^\frac12 \,ds\\
&\leq  C(K)\,\|\omega^\epsilon_0\|_\infty\,t\,|\supp\omega^\epsilon_0|^\frac12\,\max_{0\leq s\leq t}\|\det DX^\epsilon(s,\cdot)^\frac12\|_{L^\infty(\supp\omega^\epsilon_0)}.
\endaligned$$
Above, the maximum term on the right hand side  (even depending on $t$ and $\epsilon$) is finite and stays bounded as $|z|\to\infty$, due to the smoothness in $t$ and $z$ of $X^\epsilon(t,z)$. Thus, for every fixed $t$ and $\epsilon>0$ one has
\begin{equation}\label{limit}
\lim_{|z|\to\infty}\frac{|X^\epsilon(t,z)-z|}{|z|}=0.\end{equation}
As a consequence, \eqref{limit} tells us that necessarily $a^\epsilon(t)=1$, and so (a) follows. To see (b), we observe that $X^\epsilon(t,\cdot)-b^\epsilon(t)$ is a  \emph{principal} $\mathcal{K}_t$-quasiconformal map, because 
$$
|X^\epsilon(t,z)-b^\epsilon(t)-z|= O(1/|z|)
$$
as $|z|\to\infty$. Moreover, it is conformal outside of $\supp\omega_0^\epsilon$ by Lemma \ref{support}. Hence, by \cite[Theorem 13.1.2]{AIS}, we have the following area distortion estimates,
$$
|X^\epsilon(t,E)|=|X^\epsilon(t,E)-b^\epsilon(t)|\leq \mathcal{K}_t\,|E|
$$
$\forall E\supset\,\supp(\omega^{\epsilon}_{0})$, as claimed.
\end{proof}

\noindent
We are now in position of getting the $L^\infty$ bounds for $\v^\epsilon$.

\begin{prop}\label{normstability}
Assume that $K(z)=\frac{e^{i\theta}}{2\pi z}$. If $\omega_0$ is compactly supported, then 
\begin{itemize}
\item[(a)] $\|\omega^\epsilon(t,\cdot)\|_\infty\leq \|\omega_0\|_\infty$
\item[(b)] $\|\omega^\epsilon(t,\cdot)\|_1\leq \|\omega _0\|_\infty\,e^{t\|\omega_0\|_\infty}\,|\supp\omega^\epsilon_0|$.
\item[(c)] $\|\bv^\epsilon(t,\cdot)\|_\infty\leq C(K) \,e^{\frac{t}2\,\|\omega_0\|_\infty}\,\|\omega_0\|_\infty\,|\supp\omega_0^\epsilon|^\frac12$.
\end{itemize}
\end{prop}

\begin{proof}
Claim (a) can be proved by recalling that $\omega^{\epsilon}(t,\cdot)\circ X^{\epsilon}(t,\cdot) =\omega_0^{\epsilon}(\cdot)$ and the facts that  $X^\epsilon(t,\cdot)$ preserves Lebesgue-null sets and $\|\omega_0^\epsilon\|_\infty\leq\|\omega_0\|_\infty$. 
For (b), we use Lemmas \ref{qcflows} (a), \ref{support} and \ref{principal} (b) to obtain
$$
\aligned
\|\omega^\epsilon(t,\cdot)\|_1
&=\int_{\C}|\omega^\epsilon(t,z)|\,dA(z)\\
&=\int_{\C}|\omega^\epsilon_0(\zeta)|\,J^\epsilon(t,\zeta)\,dA(\zeta)\\
&\leq \|\omega^\epsilon_0\|_\infty\,\int_{\supp \omega^\epsilon_0} J^\epsilon(t,\zeta)\,dA(\zeta)\\
&= \|\omega^\epsilon_0\|_\infty\,|X^\epsilon(t,\supp \omega^\epsilon_0)|\\
&\leq \|\omega^\epsilon_0\|_\infty\,\mathcal{K}_t \,|\supp\omega^\epsilon_0|\\
&\leq \|\omega_0\|_\infty\,e^{t\|\omega_0\|_\infty}\,|\supp\omega^\epsilon_0|
\endaligned
$$
as desired.  Estimate (c) follows from Lemma \ref{LinftyK} (b).
\end{proof}

\noindent
The control on $\|\v^\epsilon\|_{L^1(\R, L^\infty)}$ allows for local boundeness of $X^\epsilon(t,\cdot)$, since 
$$
|X^\epsilon(t,z)-z|\leq\int_0^t|\v^\epsilon(s,X^\epsilon(s,z))|\,ds\leq \|\v^\epsilon\|_{L^1((0,t),L^\infty)}\leq C(K)\,|\supp\omega_0|^\frac12\,e^{\frac{t}2\,\|\omega_0\|_\infty}$$
and so Lemma \ref{compactness} guarantees the existence of a limit flow map $X(t,\cdot):\C\to\C$ which is $\mathcal{K}_t$-quasiconformal at each time $t$. Setting then $\omega(t,\cdot)=\omega_0(X(-t,\cdot))$, we obtain a well defined $L^1((0,t),L^\infty)$ function. We also define $\v(t,\cdot)=K\ast\omega(t,\cdot)$.

\begin{theo}\label{convergence}
With the above notation,
\begin{enumerate}
\item[(a)]  $\|\omega^\epsilon(t,\cdot)-\omega(t,\cdot)\|_1\to 0$, and 
\item[(b)]  $\|\bv^\epsilon(t,\cdot)-\bv(t,\cdot)\|_\infty\to 0$ .
\end{enumerate}
\end{theo}
\begin{proof}
One has
$$
\|\omega^\epsilon(t,\cdot)-\omega(t,\cdot)\|_1\leq
\|\omega_0^\epsilon(X^\epsilon(-t,\cdot))-\omega_0(X^\epsilon(-t,\cdot)\|_1+\|\omega_0 (X^\epsilon(-t,\cdot)-\omega_0(X (-t,\cdot)\|_1
$$
At the first term, we consider a disk $D$ such that $\supp\omega^\epsilon_0,\supp\omega_0\subset D$, and use the higher integrability of quasiconformal jacobians. If $1<p<\frac{\mathcal{K}_t}{\mathcal{K}_t-1}$, 
$$\aligned
\|\omega_0^\epsilon(X^\epsilon(-t,\cdot))-\omega_0(X^\epsilon(-t,\cdot)\|_1
&=\int|\omega_0^\epsilon-\omega_0|\,J^\epsilon(t,\cdot)\\
&\leq\|\omega_0^\epsilon-\omega_0\|_{L^{p'}(D)}\,\|J^\epsilon(t,\cdot)\|_{L^{p}(D)}\\
&\leq\|\omega_0^\epsilon-\omega_0\|_{L^{1}(D)}^\frac1{p'}\,\|\omega_0^\epsilon-\omega_0\|_{L^{\infty}(D)}^\frac1p\,\|J^\epsilon(t,\cdot)\|_{L^{p}(D)}
\endaligned$$
Above, $\|\omega_0^\epsilon-\omega_0\|_{L^{1}(D)}$ converges to $0$, while $\|J^\epsilon(t,\cdot)\|_{L^{p}(D)}$ is bounded in terms of $\mathcal{K}_t$ and $|D|$, 
$$
\aligned
\|J^\epsilon(t,\cdot)\|_{L^{p}(D)}
&=\left(\int_D J^\epsilon(t,\cdot)^{p}\right)^\frac1{p}\\
&\leq C(p, \mathcal{K}_t)\, |D|^{\frac1{p}-1}\int_D J^\epsilon(t,\cdot) \\
&\leq C(p, \mathcal{K}_t)\, |\supp\omega_0^\epsilon|^{\frac1{p} }  \leq C(p, \mathcal{K}_t)\, |\supp\omega_0 |^{\frac1{p} } 
\endaligned
$$
by Lemma \ref{principal} (b) and the reverse H\"older property of quasiconformal jacobians \cite{AIS}, and provided that $\epsilon>0$ is small enough. Concerning the second term, let us choose $\omega_0^n\in C_0$ such that $\|\omega_0^n-\omega_0\|_1\leq 1/n$ and $\|\omega_0^n\|_\infty\leq\|\omega_0\|_\infty$. Then
$$
\aligned
\|\omega_0 (X^\epsilon(-t,\cdot))-\omega_0 (X (-t,\cdot))\|_1
&\leq  \|\omega_0 (X^\epsilon(-t,\cdot))-\omega_0^n(X^\epsilon(-t,\cdot))\|_1\\
&+\|\omega_0^n(X^\epsilon(-t,\cdot))-\omega_0^n(X (-t,\cdot))\|_1\\
&+\|\omega_0^n(X (-t,\cdot))-\omega_0 (X (-t,\cdot))\|_1
\endaligned
$$
Above, again because of the higher integrability of quasiconformal jacobians,
$$\aligned
\|\omega_0(X^\epsilon(-t,\cdot))-\omega_0^n(X^\epsilon(-t,\cdot))\|_1
&=\int  |\omega_0 -\omega_0^n |\,J^\epsilon(t,\cdot)\\
&=\|\omega_0 -\omega_0^n \|_{L^{p'}(D)}\,\|J^\epsilon(t,\cdot)\|_{L^p(D)}\\
&\leq\|\omega_0 -\omega_0^n \|_{L^{1}(D)}^\frac1{p'}\,\|\omega_0^n-\omega_0\|_\infty^\frac1p\,\|J^\epsilon(t,\cdot)\|_{L^p(D)}\\
&\leq n^\frac{-1}{p'}\,2^{\frac{1}{p}}\|\omega_0\|_\infty^\frac1p\,\|J^\epsilon(t,\cdot)\|_{L^p(D)}
\endaligned$$
and similarly for $\|\omega_0^n(X (-t,\cdot))-\omega_0(X (-t,\cdot))\|_1$. Thus each of these two terms can be made smaller than $\delta/3$ if $n$ is chosen large enough. The control of the second term comes by continuity. Precisely, as $X^\epsilon(-t,\cdot)\to X(-t,\cdot)$ and $\omega_0^n$ is continuous, there is $\epsilon>0$ such that $\|\omega_0^n(X^\epsilon(-t,\cdot)-\omega_0^n(X(-t,\cdot))\|_\infty<\delta/3$. Thus (a) follows.  For the proof of (b), use (a) and Lemma \ref{LinftyK} (b). 
\end{proof}

\noindent
The above convergence result suffices to prove that $\omega$ is a weak solution to the desired nonlinear transport equation. Existence is proved.\\
\\
As we said in the introduction, Reimann's Theorem \ref{Reimannth} extends (as proven in \cite[Theorem 1]{CJ}) to vector fields $\v$ such that $$\overline\partial \v + \lambda\,\Im(\partial \v)\in L^\infty,$$
that is, the flow $X(t,\cdot)$ of these vector fields consists of quasiconformal mappings. Above, one may choose $\lambda\in\C$ to be a constant with $|\lambda|<1$, or also a smooth, compactly supported function $\lambda\in C^\infty_c(\C)$ with $\|\lambda\|_{L^\infty(\R^2)}<1$. This makes it reasonable to look for extensions of Theorem \ref{main} to other kernels $K(z)$ different than the one we used here $K(z)=\frac{e^{i\theta}}{2\pi z}$. The new kernels $K$ we have in mind are complex  multiples of the fundamental solution of the operator $\overline\partial \v + \lambda\,\Im(\partial \v)$.

\section{The velocity formulation}\label{governing}

Let us recall that the Euler's system of equations is given, in its original formulation, in terms of the velocity field $\v$. Namely, one has the following equivalence
$$
\begin{cases}
\omega_t+\v\cdot\nabla\omega=0,\\\v=\frac{i}{2\pi\bar{z}}\ast\omega,\\\omega|_{t=0}=\omega_0
\end{cases}
\hspace{1cm}
\Longleftrightarrow
\hspace{1cm}
\begin{cases}
\v_{t}+\v\cdot\nabla\v=-\nabla p,\\\div\v=0,\\\curl\v|_{t=0}=\frac{1}{2 } \omega_0.
\end{cases}
$$
where $p$ is the scalar valued \emph{pressure} function. It turns out that a similar equivalent formulation can be provided for \eqref{cauchysystem}, and this is our goal in the present section.  From now on, we denote
$$\bC=\left(\begin{array}{cc}1&0\\0&-1\end{array}\right)$$
and set $\bM_\theta z = e^{i\theta}\,\bC\,z=e^{i\theta}\,\bar{z}.$ Thus, indeed $\bM_\theta$ is the $\R$-linear map with matrix 
$$\bM_\theta=\left(\begin{array}{cc}\cos\theta&\sin\theta\\\sin\theta&-\cos\theta\end{array}\right)$$
To avoid formalities, we reduce ourselves to the smooth setting, and assume the datum $\omega_0:\C\to\R$ is smooth and compactly supported. Let us remind that $K(z)=K_\theta(z)=\frac{e^{i\theta}}{2\pi z}$.

\begin{prop}\label{equivform}
The scalar-valued function $\omega:[0,T]\times\C\to\R$ is a weak solution of
\begin{equation}\label{omegaequation}
\begin{cases} \omega_t+\bv\cdot\nabla\omega=0\\\bv=K\ast\omega \\\omega|_{t=0}=\omega_0\end{cases}
\end{equation}
if and only if $\bv:[0,T]\times\C\to\C$ and $q:[0,T]\times\C\to\R$ solve
\begin{equation}\label{vequation}
\begin{cases}\bv_t+\bv\cdot\nabla\bv = -\bM_\theta\nabla q\\
-\Delta q = \div(\bv)\,\div(\bM_\theta\bv)\\
\curl(\bM_\theta\bv)|_{t=0}=0\\
\div(\bM_\theta\bv)|_{t=0}=\omega_0
\end{cases}
\end{equation}
also in the weak sense.
\end{prop}
\begin{proof}
We first go from \eqref{vequation}  to \eqref{omegaequation}. We identify $\R^2\equiv \C$, and write the system \eqref{vequation} in complex notation,
$$
\begin{cases}
\v_t+\v\,\partial\v+\overline\v\,\overline\partial\v = -e^{i\theta}\overline{\nabla q}\\-\Delta q = \div(\v)\,\div(e^{i\theta}\bar\v)\\
\Im(\partial(e^{i\theta}\bar\bv))|_{t=0}=0\\
\Re(\partial(e^{i\theta}\bar\v))|_{t=0}=\omega_0\end{cases}
$$
Now, taking $\overline\partial$ on the first equation, and obtain
$$
(\overline\partial\v)_t+\v\,\partial(\overline\partial\v)+\overline\v\,\overline\partial(\overline\partial\v )+\overline\partial \v (\partial\v+\overline\partial\overline\v)= -\overline\partial(e^{i\theta}\overline{\nabla q})
$$
or equivalently,
$$
(\overline\partial\v)_t+\v\cdot\nabla(\overline\partial\v)+ \overline\partial\v \,\div\v= -\frac12\,e^{i\theta}\,\Delta q.
$$
We now multiply by $e^{-i\theta}$, and use the $\C$-linearity of the transport operator $\frac{d}{dt}+\v\cdot\nabla$ to get
$$
(e^{-i\theta}\,\overline\partial\v)_t+\v\cdot\nabla(e^{-i\theta}\,\overline\partial\v)+ e^{-i\theta}\,\overline\partial\v \,\div\v= -\frac12\,\Delta q.
$$
After taking real and imaginary parts, 
\begin{equation}\label{realimparts}
\begin{cases}
\Re ((e^{-i\theta}\overline\partial\v)_t+\v\cdot\nabla(e^{-i\theta}\overline\partial\v))+\Re(e^{-i\theta}\overline\partial\v)\,\div\v = -\frac12\,\Delta q\\
\Im((e^{-i\theta}\overline\partial\v)_t+\v\cdot\nabla(e^{-i\theta}\overline\partial\v))+\Im(e^{-i\theta}\overline\partial\v)\,\div\v= 0.
\end{cases}
\end{equation}
The above equations may be seen as scalar conservation laws for $\Re(e^{-i\theta}\,\overline\partial\v)$ and $\Im(e^{-i\theta}\,\overline\partial\v)$. The second one is homogeneous, and so from the initial condition
$$2\,\Im(e^{-i\theta}\overline\partial\v)|_{t=0}=-\curl(\bM_\theta\v)|_{t=0}=0$$ 
we deduce that at any time $t>0$  
$$2\,\Im(e^{-i\theta}\overline\partial\v)=-\curl(\bM_\theta\v)=0.$$ 
To see this, simply call $\rho=2\,\Im(e^{-i\theta}\overline\partial\v)$ and note it satisfies the following initial value problem, 
$$\begin{cases}
\frac{d}{dt}\rho+\div(\rho\,\v)=0\\\rho(0,\cdot)=0
\end{cases}$$
which has $\rho=0$ as its unique solution, due to the smoothness of $\v$. As a consequence, $e^{-i\theta}\overline\partial\v\in\R$ and so if we now denote $\omega=2\Re(e^{-i\theta}\overline\partial\v)$, then 
$$\omega = \div(e^{i\theta}\bar\v).$$ 
Thus the first equation at \eqref{realimparts} implies that 
$$\omega_t+\v\cdot\nabla\omega+\omega\,\div\v= -\Delta q.$$
Now, since the second equation at \eqref{vequation} tells us that $\omega\,\div\v= -\Delta q$, we necessarily have for $\omega$ a homogeneous transport equation
$$\omega_t+\v\cdot\nabla\omega=0$$
together with the initial condition $\omega|_{t=0}=\div(e^{i\theta}\bar\v)|_{t=0}=\omega_0$ as claimed. \\
\\
For the converse implication, we start by noting that our choice of the kernel $K$ and the second equation in \eqref{omegaequation} tell  us that $2e^{-i\theta}\overline\partial\bv=\omega$, which by assumption is real valued. We now use the first equation in \eqref{omegaequation}, together with the $\C$-linearity of the complex operator, to get
$$
\overline\partial\v_t + \v\cdot\partial(\overline\partial \v)+\overline{\v}\cdot\overline\partial(\overline\partial\v)=0
$$
or equivalently
$$
\overline\partial(\v_t +  \v \cdot\partial \v+\overline{\v}\cdot\overline\partial\v)=  \overline\partial\v \,\div\v
$$
We now complex conjugate at both sides of the equality, multiply by $e^{i\theta}$, and use $\C$-linearity of the transport operator, and obtain
\begin{equation}\label{aqui}
\partial(e^{i\theta}(\overline{\v_t +  \v \cdot\partial \v+\overline{\v}\cdot\overline\partial\v}))=\frac\omega2\,\div\v
\end{equation}
By assumption, the right hand side above is real, whence $e^{i\theta}(\overline{\v_t +  \v \cdot\partial \v+\overline{\v}\cdot\overline\partial\v})$ is a conservative vector field. Thus there exists a scalar valued potential $q$ such that
$$
e^{i\theta}(\overline{\v_t +  \v \cdot\partial \v+\overline{\v}\cdot\overline\partial\v})=-\nabla q
$$ 
This automatically gives the first equation at \eqref{vequation}. Moreover, if we take real parts at \eqref{aqui},
$$-\frac12\Delta q=\frac12\div(e^{i\theta}(\overline{\v_t +  \v \cdot\partial \v+\overline{\v}\cdot\overline\partial\v}))=\frac\omega2\,\div\v$$
or equivalently
$$-\Delta q= \div(\v)\,\div(\bM_\theta \v)$$
as claimed. The third and fourth equations in \eqref{vequation} are automatic from the second and fourth in \eqref{omegaequation}.
\end{proof}

\noindent
One of the initial motivations to find the above equivalent formulation was to look for energy estimates that help in proving that the problem \eqref{cauchysystem} has indeed a unique solution, exactly in the same way uniqueness for Euler's is proven in \cite[pp. 320-321]{BM}. Unfortunately, and in contrast to the case of Euler's incompressible system, the velocity formulation does not seem to help in proving uniqueness. Indeed, if $\v^1,\v^2$ are two solutions to \eqref{cauchysystem} with the same datum $\omega_0$, then the difference $\v=\v^1-\v^2$ may fail to belong to $L^2$, but it certainly belongs to $L^p$ for any $p>2$. Thus the energy $E(t)=E_p(t)=\frac1p\,\|\v\|_p^p$ is well defined. Moreover, the velocity formulation \eqref{vequation} provides us with the estimate
$$E'(t)\leq C_1\,E(t)+C_2\,q\,E(t)^{1-\frac1q}+C_3\,E(t)^{1-\frac1p}$$
for all large values of $q$, and where the constants $C_1, C_2, C_3$ are independent of $q$ and $t$. In the very special Euler's setting, the value $p=2$ is allowed, and the velocity equation provides for $E_2(t)$ a similar inequality with $C_1=C_3=0$, which immediately implies uniqueness (as it forces $E(t)=0$ for $t>0$). In our setting, though, the presence of the $C_3$ term explicitly breaks the argument. Thus new ideas seem to be needed for proving uniqueness for bounded solutions of \eqref{cauchysystem}.

\vspace{2cm}
 \noindent
 Albert Clop\\
 Department of Mathematics and Computer Science\\
 Universitat de Barcelona\\
 08007-Barcelona\\
 CATALONIA\\
 albert.clop@ub.edu\\
 \\
 \\
 Banhirup Sengupta\\
 Departament de Matem\`atiques\\
 Universitat Aut\`onoma de Barcelona\\
 08193-Bellaterra\\
 CATALONIA\\
 sengupta@mat.uab.cat\\

\end{document}